\documentclass[11pt]{amsart}

\usepackage{array,float}
\usepackage{amssymb, amsmath, amsthm, amsbsy, amscd, mathrsfs}

\usepackage[normalem]{ulem}
\numberwithin{equation}{section}
\numberwithin{figure}{section}

\theoremstyle{plain}
 \newtheorem{theorem}{Theorem}[section]
 \newtheorem{proposition}[theorem]{Proposition}
 \newtheorem{lemma}[theorem]{Lemma}
 \newtheorem{corollary}[theorem]{Corollary}
 
 \newtheorem{thmABC}{Theorem}

\theoremstyle{definition}
 \newtheorem{definition}[theorem]{Definition}

\theoremstyle{remark}
 
 \newtheorem{remark}[theorem]{Remark}
 \newtheorem*{acknowledgements}{Acknowledgements}
 
 \newtheorem{example}[theorem]{Example}

\newcommand{\F}{\ensuremath{\mathbb{F}}}
\newcommand{\Fp}{\ensuremath{\mathbb{F}_p}}
\newcommand{\Fq}{\ensuremath{\mathbb{F}_q}}
\newcommand{\Fpf}{\ensuremath{\mathbb{F}_{p^f}}}
\newcommand{\N}{\ensuremath{\mathbb{N}}}
\newcommand{\Zp}{\ensuremath{\mathbb{Z}_p}}
\newcommand{\pP}{\ensuremath{\mathbb{P}}}
\newcommand{\C}{\ensuremath{\mathbb{C}}}
\newcommand{\R}{\ensuremath{\mathbb{R}}}
\newcommand{\Q}{\ensuremath{\mathbb{Q}}}
\newcommand{\Z}{\ensuremath{\mathbb{Z}}}

\newcommand{\bfe}{\ensuremath{\mathbf{e}}}
\newcommand{\bff}{\ensuremath{\mathbf{f}}}

\newcommand{\Frcq}{F_{r,c}(\mathbb{F}_q)}
\newcommand{\frcq}{\mathfrak{f}_{r,c}(\mathbb{F}_q)}

\newcommand{\bfx}{{\boldsymbol{x}}}
\newcommand{\bfy}{\ensuremath{\mathbf{y}}}

\newcommand{\bfX}{\ensuremath{\mathbf{X}}}
\newcommand{\bfY}{\ensuremath{\mathbf{Y}}}

\newcommand{\mff}{\ensuremath{\mathfrak{f}}}
\newcommand{\mfg}{\ensuremath{\mathfrak{g}}}
\newcommand{\mfh}{\ensuremath{\mathfrak{h}}}

\newcommand{\mfgp}{\ensuremath{\mathfrak{g}'}}

\newcommand{\mfz}{\ensuremath{\mathfrak{z}}}

\newcommand{\rf}{\ensuremath{{\bf k}}}

\newcommand{\fieldK}{\ensuremath{{\mathbb K}}}

\newcommand{\rk}{\ensuremath{{\rm rk}}}

\newcommand{\wt}{\ensuremath{\widetilde}}
\newcommand{\wtB}{\ensuremath{\widetilde{B}}}

\newcommand{\lri}{\mathfrak o}
\newcommand{\Lri}{\mathfrak O}

\newcommand{\mfp}{\mathfrak{p}}
\newcommand{\mfP}{\mathfrak{P}}
\newcommand{\fS}{\mathcal{S}}
\newcommand{\hallgenerators}{\Delta}
\newcommand{\hallbasis}{\mathcal{H}}
\newcommand{\wh}{\widehat}
\newcommand{\rarr}{\rightarrow}
\newcommand{\kk}{\ensuremath{\mathbf{k}}}

\newcommand{\udots}{\mathinner{\mskip1mu\raise1pt\vbox{\kern7pt\hbox{.}}
\mskip2mu\raise4pt\hbox{.}\mskip2mu\raise7pt\hbox{.}\mskip1mu}}

\DeclareMathOperator{\Pf}{Pf}
\DeclareMathOperator{\Hom}{Hom}

\DeclareMathOperator{\Id}{Id}
\DeclareMathOperator{\Rad}{Rad}
\DeclareMathOperator{\Stab}{Stab}

\DeclareMathOperator{\Ad}{Ad}
\DeclareMathOperator{\Ind}{Ind}
\DeclareMathOperator{\cc}{cc}
\DeclareMathOperator{\ad}{ad}
\DeclareMathOperator{\im}{im}
\DeclareMathOperator{\Mat}{Mat}
\DeclareMathOperator{\weight}{wt}

\DeclareMathOperator{\dl}{dl}
\DeclareMathOperator{\cd}{cd}
\DeclareMathOperator{\cs}{cs}
\DeclareMathOperator{\ch}{ch}

\DeclareMathOperator{\trans}{tr}
\renewcommand{\epsilon}{\varepsilon}
\renewcommand{\phi}{\varphi}
\renewcommand{\mid}{:}

\begin{document}

\title[Enumerating conjugacy classes and characters of
$p$-groups]{Enumerating classes and characters of $p$-groups}

\author{E.~A.~O'Brien}
\author{C.~Voll}

\address{E.~A.~O'Brien, Department of Mathematics, University of
  Auckland, Auckland, New Zealand} \email{obrien@math.auckland.ac.nz}
\address{C.~Voll, School of Mathematics, University of Southampton,
  University Road, Sou\-thampton SO17 1BJ, United Kingdom}
\curraddr{Fakult\"at f\"ur Mathematik, Universit\"at
  Bielefeld\\ Postfach 100131\\ D-33501 Bielefeld\\Germany}
\email{C.Voll.98@cantab.net}


\begin{abstract} 
  We develop general formulae for the numbers of conjugacy classes and
  irreducible complex characters of finite $p$-groups of nilpotency
  class less than $p$. This allows us to unify and generalize a number
  of existing enumerative results, and to obtain new such results for
  generalizations of relatively free $p$-groups of exponent~$p$.  Our
  main tools are the Lazard correspondence and the Kirillov orbit
  method.
\end{abstract}

\maketitle

\setcounter{tocdepth}{3}

\thispagestyle{empty}

\section{Introduction}

The study of the conjugacy classes and irreducible complex characters
of groups is an active area of research.  The enumeration of classes
and characters of finite groups of Lie type, for instance, has played
an important role in the work of Liebeck, Shalev and others; see, for
instance, \cite{LiebeckShalev/05}.  Motivated by a conjecture of
Higman \cite{HigmanI/60}, the classes and characters of
upper-unitriangular groups have been extensively studied; see, for
example, \cite{Isaacs/07, VeralopezArregi/03}.

`Representation growth of groups' is an umbrella term for the
asymptotic and arithmetic properties of group representations as a
function of their dimensions.  A key tool in the study of
representation growth is the Kirillov orbit method. Where applicable,
it provides a parameterization of the irreducible complex
representations of a group in terms of co-adjoint orbits. It was
pioneered by Kirillov in the realm of nilpotent Lie groups and later
adapted to other classes of groups, including $p$-adic analytic
groups, finitely generated nilpotent groups, and finite $p$-groups;
see~\cite{Gonzalez/09, Howe-compact/77, Howe-nilpotent/77, Jaikin/06}.
Under certain conditions the linearization achieved by this method
facilitates a description of the numbers of characters of a group in
terms of geometric data attached to the dual of a Lie algebra
associated with the group, such as the numbers of rational points of
certain algebraic subvarieties.

 Let $p$ be a prime. In this paper we employ the Kirillov orbit method
 to study the classes and characters of finite $p$-groups of
 nilpotency class less than~$p$. Let $G$ be a finite $p$-group. For
 $i\geq 0$, we define
\begin{align*}
\cc_i(G) & = \#\{\text{conjugacy classes of $G$ of cardinality
$p^i$}\}\text{ and }\\
\ch_i(G) & = \#\{\text{irreducible complex characters of $G$ of degree
$p^i$}\}.
\end{align*}
The vectors $\cc(G) = (\cc_i(G))_i$ and $\ch(G) = (\ch_i(G))_i$ are
the \emph{class vector} and the \emph{character vector} of~$G$,
respectively. We denote by $\cs(G)= \{p^i \mid \cc_i(G)\neq 0\}$ the
\emph{class sizes} of $G$ and by $\cd(G) = \{p^i \mid \ch_i(G)\neq
0\}$ the \emph{character degrees} of $G$.  We write $$k(G)=\sum_i
\cc_i(G) = \sum_i \ch_i(G)$$ for the \emph{class number} of~$G$.

Let $c$ be the nilpotency class of $G$, and assume that~$c<p$. Let
$\mfg=\log(G)$ be the finite Lie ring associated to $G$ by the Lazard
correspondence. We associate to $G$ a subset $\fS(G)$ of $\mfg/\mfz
\times \widehat{\mfg'}$, where $\mfz$ denotes the centre of $\mfg$ and
$\widehat{\mfg'}=\Hom_{\Z}(\mfg',\C^\times)$ the Pontryagin dual of
the derived Lie ring~$\mfg'$. In Theorem~\ref{theorem A} we show that
the class and conjugacy vectors of $G$ may be described in terms of
the cardinalities of fibres of the natural projections from $\fS(G)$
onto $\mfg/\mfz$ and~$\widehat{\mfg'}$.

Theorem~\ref{theorem B} gives a geometric description of the class and
character vectors of certain $p$-groups and describes the variation of
these vectors under `extension of scalars'. More precisely, let $\lri$
be a compact discrete valuation ring of characteristic zero with
residue field $\rf$ of characteristic $p$.  Theorem~\ref{theorem B}
asserts that if $\mfg$ is a finite, nilpotent $\lri$-Lie algebra of
class $c < p$, and $\mfg'$ or, equivalently, $\mfg/\mfz$, is a
$\kk$-vector space, then computing class and character vectors of the
$p$-group $\exp(\mfg)$ associated to $\mfg$ under the Lazard
correspondence is equivalent to enumerating $\rf$-rational points of
degeneracy loci of certain `commutator matrices' associated with
$\mfg$.  Moreover, the formulae given in Theorem~\ref{theorem B} are
uniformly valid for groups of the form $\exp(\mfg \otimes_\lri\Lri)$,
where $\Lri$ is a finite, unramified extension of $\lri$.

The Lie algebra $\mfg$ may be obtained by base change from a globally
defined object, such as a nilpotent $\Z$-Lie algebra. For some of the
groups obtained from such Lie algebras, Theorem~\ref{theorem B} yields
formulae which are uniform under variation of both the cardinality and
the characteristic of the residue field. Consider, for instance, the
free $\Fq$-Lie algebras $\frcq$ on $r$ generators and of nilpotency
class $c$, where $\Fq$ is a finite field of
characteristic~$p>c$. These algebras are of the form $\frcq =
\mathfrak{f}_{r,c}(\Z)\otimes_\Z \Fq$, where $\mathfrak{f}_{r,c}(\Z)$
is the free nilpotent $\Z$-Lie algebra of class $c$ on $r$
generators. Theorem~\ref{theorem B} applies to the groups $\Frcq :=
\exp(\frcq)$.

\medskip

In Section~\ref{two} we state Theorems~\ref{theorem A}
and~\ref{theorem B}, together with some applications to groups of the
form~$\Frcq$.  Our main tools are the Lazard correspondence for
$p$-groups of nilpotency class~$c<p$ and the Kirillov orbit method for
such groups. In Section \ref{three} we review these tools and use them
to prove Theorems~\ref{theorem A} and~\ref{theorem B}.  In Section
\ref{four} we apply these results to uniformize a number of existing
enumerative results on classes and characters of $p$-groups.  In
Section \ref{five} we prove new results for the groups $\Frcq$,
including those stated in Section~\ref{two}.  They extend and
generalize results of Ito and Mann~\cite{ItoMann/06} for the
relatively free groups of exponent~$p$.

\subsection{Notation}
We denote the cardinality of a set $S$ by either $\#S$ or $|S|$. We
write $\N$ for the set $\{1,2,\dots\}$ of natural numbers. For
$I\subseteq\N$ and $c\in\R$, we write $I_0$ for $I\cup\{0\}$ and
$cI_0$ for $\{ci \mid i\in I_0\}$. Given $a,b\in\N_0$ we define
$[a]=\{1,\dots,a\}$ and $[a,b] = \{a,\dots,b\}$. For $x\in\R$ we set
$\lfloor x \rfloor := \max\{m\in\Z \mid m \leq x\}$.  If $I$ is any
ordered set then we write $I = \{i_1,\dots,i_l\}_<$ to indicate that
$i_1<\dots<i_l$.  Given a proposition $P$, the `Kronecker delta'
$\delta_P$ is $1$ if $P$ holds and $0$ otherwise.  If
$n_1,\dots,n_r\in\N_0$ and $f\in\N$, we write $(n_1,\dots,n_r)_f$ for
the vector
$$(n_1,\underbrace{0,\dots,0}_{f-1},n_2,\underbrace{0,\dots,0}_{f-1},\dots,n_r,\underbrace{0,\dots,0}_{f-1})\in\N_0^{fr};$$
if $f=1$ we drop the subscript.

Given a ring $R$, an $R$-Lie algebra $\mfg$ is an $R$-algebra with a
`Lie bracket', that is to say an $R$-bilinear map $[\,,]:\mfg\times
\mfg\rightarrow \mfg$ which is skew-symmetric and satisfies the Jacobi
identity. A Lie ring is a $\Z$-Lie algebra.  We write $[u,v,w,\dots]$
for the left-normed Lie product $[\dots[[u,v],w]\dots] \in L$, and
$[u,_iv]$ denotes the Lie product $[u,v,\dots,v]$ with~$i$ occurrences
of~$v$.

Throughout this paper, $\lri$ is a compact discrete valuation ring of
characteristic zero, viz.\ a finite extension of the $p$-adic
integers~$\Zp$, with maximal ideal~$\mfp$ and residue field $\kk =
\lri/\mfp$ of characteristic~$p$. An arbitrary field is denoted
by~$\fieldK$.

The centre and derived group of a group $G$ are denoted by $Z$ (or
$Z(G)$) and $G'$ respectively; the centre and derived ring of a Lie
algebra $\mfg$ are $\mfz$ (or $Z(\mfg)$) and $\mfg'$. We write $[\,,]$
also for the induced map $\mfg/\mfz\times \mfg/\mfz \rarr \mfg'$,
$(x+\mfz,y+\mfz)\mapsto [x,y]$. Given $g\in G$ and $x\in\mfg$ we write
$C_G(g)$ and $C_\mfg(x)$ for the respective centralizers.

Given a ring $R$ and integers $m$ and $n$, we write $\Mat(n\times
m,R)$ for the $n\times m$-matrices over $R$. We abbreviate
$\Mat(n\times n,R)$ to $\Mat(n,R)$. We denote the transpose of a
matrix $A$ by $A^{\trans}$.

By a character of a group we always mean a complex irreducible
character.

\section{The main results}
\label{two}

The Lazard correspondence establishes an order-preserving one-to-one
correspondence between finite $p$-groups of nilpotency class $c<p$ on
the one hand and finite nilpotent Lie rings of $p$-power order and
class $c<p$ on the other; cf.~\cite[Example~10.24]{Khukhro/98}. More
precisely, one may define a group operation on such a Lie ring $\mfg$
by the formula
$$u\star v := \sum_{i\leq c}F_i(u,v),\quad u,v\in\mfg,$$ where
$F_i(X,Y)$ is the homogeneous part of degree $i$ of the Haus\-dorff
series $F(X,Y)$, an element in the completion of the free $\Q$-Lie
algebra on variables $X$ and $Y$; cf.~\cite[\S 9.2]{Khukhro/98}.  Then
$\exp(\mfg):=(\mfg,\star)$ is a $p$-group of class~$c$. The theorem
underlying the Lazard correspondence asserts that the isomorphism type
of every $p$-group $G$ of class $c<p$ arises in this manner from a Lie
ring~$\mfg$, unique up to isomorphism. We denote the map underlying a
fixed isomorphism $\exp(\mfg)\cong G$ by $\exp:\mfg\rightarrow G$, and
write $\log$ for its inverse.  We write $\widehat{\mfg'}$ for the
Pontryagin dual $\Hom_{\Z}(\mfg',\C^\times)$ of the finite abelian
$p$-group~$\mfg'$.

\begin{thmABC}\label{theorem A} 
Let $G$ be a finite $p$-group of nilpotency class~$c<p$ and let
$\mfg=\log(G)$ be the corresponding Lie ring. Define
\begin{equation*}
\fS(G) :=\{(x,\omega)\in\mfg/\mfz\times\widehat{\mfg'} \mid
\omega([x,z])=1 \emph{ for all }z\in\mfg/\mfz\},\label{def:S}
\end{equation*}
with projections $\pi_1:\fS(G)\rightarrow\mfg/\mfz$ and
$\pi_2:\fS(G)\rightarrow\widehat{\mfg'}$. For $i \geq 0$,
\begin{align*}
\cc_i(G)&=\#\left\{x\in\mfg/\mfz \mid
|\pi_1^{-1}(x)|=p^{-i}|\widehat{\mfg'}|\right\}|Z(G)|p^{-i},\\
\ch_i(G)&=\#\left\{\omega\in\widehat{\mfg'} \mid
|\pi_2^{-1}(\omega)|=p^{-2i}|\mfg/\mfz|\right\}|G/G'|p^{-2i}.
\end{align*}
In particular, the class number $k(G) = |\fS(G)|\; |Z(G)|\,|G'|^{-1}$.
\end{thmABC}

For a certain family of groups, Theorem~\ref{theorem B} exploits this
result to provide a uniform description of the class and character
vectors in terms of the numbers of rational points of rank varieties
of matrices of linear forms.  We now formulate this more precisely.

Assume that $\lri$ is a compact discrete valuation ring of
characteristic zero and residue characteristic $p$, and that $\mfg$ is
a finite, nilpotent $\lri$-Lie algebra of class~$c<p$.  Set
$$ a := \rk_{\lri}(\mfg/\mfz),\qquad b := \rk_{\lri}(\mfgp),$$ and fix
an ordered set $\bfe=(e_1,\dots,e_a)$ of $\lri$-module generators for
$\mfg/\mfz$ and an ordered set $\bff=(f_1,\dots,f_b)$ of $\lri$-module
generators for~$\mfgp$.  We choose `structure constants'
$\lambda_{ij}^k\in\lri$ such that
$$[e_i,e_j]=\sum_{k=1}^b\lambda_{ij}^kf_k\quad \text{ and }\quad \lambda_{ij}^k=-\lambda_{ji}^k$$ for all $i,j\in[a]$, $k\in[b]$.
\begin{definition} \label{def:commutator matrices} Let
  $\bfX=(X_1,\dots,X_a)$ and $\bfY=(Y_1,\dots,Y_b)$ be independent
  variables. We define {\it commutator matrices} (with respect to
  $\bfe$ and $\bff$) of $\lri$-linear forms in $\bfX$ and $\bfY$,
  namely
\begin{align*}
A(\bfX)\in\Mat(a\times b,\lri[\bfX]),&\text{ where
}A(\bfX)_{ik}:=\sum_{j=1}^a\lambda_{ij}^kX_j, \quad i\in[a],
k\in[b],\\ B(\bfY)\in\Mat(a,\lri[\bfY]),& \text{ where }
B(\bfY)_{ij}:=\sum_{k=1}^b\lambda_{ij}^kY_k, \quad i,j\in[a].
\end{align*}
If $\mfg$ is a $\fieldK$-algebra with $\fieldK$-basis $\mathcal{B} =
(e_1,\dots,e_h)$ such that the residue classes of the elements
$e_1,\dots,e_a$ form a $\fieldK$-basis $\bfe$ for $\mfg/\mfz$ and
$\bff = (e_{h-b+1},\dots,e_h)$ is a $\fieldK$-basis for $\mfg'$ then
we refer to the associated commutator matrices $A$ and $B$ as `with
respect to $\mathcal{B}$'.
\end{definition}

\begin{remark}\label{rem:pfaffian}
  The commutator matrix $B$ is clearly skew-symmetric. Hence
  $\det(B)$ is a square in $\lri[\bfY]$, whose square root
  $\Pf(B):=\sqrt{\det(B)}$ is the \emph{Pfaffian} of~$B$. If $a$ is
  odd then~$\Pf(B)=0$.
\end{remark}

Assume now that $\mfg/\mfz$ or, equivalently, $\mfg'$ is annihilated
by $\mfp$, the maximal ideal of~$\lri$. We write $\kk$ for the residue
field $\lri/\mfp$ of characteristic~$p$. The set of generators $\bff$
for $\mfg'$ may be regarded as a $\kk$-basis for the $\kk$-vector
space $\mfg'$. Similarly, we view $\bfe$ as a $\kk$-basis for the
$\kk$-vector space~$\mfg/\mfz$.

The commutator matrices $A$ and $B$ may be considered as matrices of
linear forms over~$\kk$. Let $\fieldK$ be an extension of $\kk$. For
$\bfx=(x_1,\dots,x_a)\in \fieldK^a$ we write $A(\bfx)\in\Mat(a\times
b,\fieldK)$ for the matrix obtained by evaluating the variables $X_i$
at $x_i$. Likewise $B(\bfy)\in\Mat(a,\fieldK)$ is defined
for~$\bfy=(y_1,\dots,y_b)\in \fieldK^b$. We note that the ranks of
matrices of the form $B(\bfy)$, for $\bfy\in \fieldK^b$, are even
integers.

It is well-known that every finite field $\kk$ is self-dual, i.e.\
(noncanonically) isomorphic to its Pontryagin dual. Indeed, let
$\psi:\kk \rightarrow \C^\times$ be a nontrivial additive character
of~$\kk$. For $a\in\kk$ define $\psi_a(x) = \psi(ax)$
for~$x\in\kk$. The map $a\mapsto \psi_a$ is an isomorphism between
$\kk$ and its Pontryagin dual~$\widehat{\kk}$; cf., for
instance,~\cite{Boyarchenko/11}. Since $\mfg'\cong \kk^b$, this yields
an isomorphism between $\mfg'$ and its dual $\wh{\mfg'}$. On the other
hand there is, of course, a -- likewise noncanonical -- isomorphism
between $\mfg'$ and its linear dual $\Hom_{\kk}(\mfg',\kk)$.  We fix
an isomorphism $\psi_1:\wh{\mfg'} \rarr \Hom_{\kk}(\mfg',\kk)$.
The dual $\kk$-basis $\bff^\vee=(f_k^\vee)$ for
$\Hom_{\kk}(\mfg',\kk)$ gives a coordinate system
\begin{equation*}
\psi_2:\Hom_{\kk}(\mfg',\kk)  \rightarrow \kk^b, \quad
y = \sum_{k=1}^b y_k f^\vee_k \mapsto \bfy = (y_1,\dots,y_b).
\end{equation*}
Set $\psi:=\psi_2\circ\psi_1:\wh{\mfg'}\rarr \kk^b$.  Similarly, the
$\kk$-basis $\bfe$ for $\mfg/\mfz$ gives a coordinate system
\begin{equation*}
  \phi:\mfg/\mfz  \rightarrow \kk^a, \quad
  x = \sum_{j=1}^a x_je_j \mapsto \bfx = (x_1,\dots,x_a).
\end{equation*}

For a finite extension $\Lri$ of $\lri$, we write $\mfg(\Lri)$ for
$\mfg\otimes_\lri\Lri$ and $\mfz(\Lri)$ for~$\mfz\otimes_\lri\Lri$.
By tensoring, the bases associated with $\mfg$ yield corresponding
bases associated with $\mfg(\Lri)$; we continue to write $\bfe$ for
the $\Lri$-basis $\bfe\otimes_\lri 1$ for $\mfg(\Lri)/\mfz(\Lri)$, and
likewise $\bff$ for the $\Lri$-basis $\bff\otimes_\lri 1$
of~$\mfg(\Lri)'$.  Note that the commutator matrices $A$ and $B$
remain unchanged.

Assume further that $\Lri$ is an unramified extension of $\lri$, with
maximal ideal~$\mfP$. We identify the residue field~$\Lri/\mfP$, a
finite extension of~$\kk$, with~$\Fq$. The $\Lri$-Lie algebra
$\mfg(\Lri)$ inherits the property that the derived algebra and the
cocentre of $\mfg(\Lri)$ are annihilated by $\mfP$. We consider $\bfe$
and $\bff$ as $\Fq$-bases for the respective $\Fq$-vector spaces of
dimensions $a$ and~$b$. Set $G(\Lri):=\exp(\mfg(\Lri))$. Note that our
assumption on $\mfg$ implies that both $G(\Lri)'$ and
$G(\Lri)/Z(G(\Lri))$ have exponent~$p$. Our second main result gives a
uniform description of the set $\fS(G(\Lri))$ introduced in
Theorem~\ref{theorem A} -- and therefore for the class and character
vectors of $G(\Lri)$ -- in terms of the numbers of $\Fq$-rational
points of degeneracy loci of the commutator matrices $A$ and~$B$.

\begin{thmABC} \label{theorem B} Let $\lri$ be a compact discrete
  valuation ring of characteristic zero and residue field $\kk$ of
  characteristic~$p$, and let $\mfg$ be a finite, nilpotent $\lri$-Lie
  algebra of class $c<p$. Assume that $\mfg'\cong\kk^b$ and that
  $\mfg/\mfz\cong\kk^a$ as $\kk$-vector spaces. Let $\Lri$ be a
  finite, unramified extension of~$\lri$, with residue field
  isomorphic to~$\Fq$.  The class sizes and character degrees of
  $G(\Lri)$ are powers of~$q=p^f$. For $i \geq 0$,
\begin{align*}
  \cc_{if}(G(\Lri)) &= \#\left\{\bfx\in \Fq^a \mid \rk({A(\bfx)}) =
  i\right\} \, |Z(G(\Lri))| q^{-i},\\ \ch_{if}(G(\Lri)) &=
  \#\left\{\bfy\in \Fq^b \mid \rk({B(\bfy)}) = {2i}\right\}\, |
  G(\Lri)/G(\Lri)' | q^{-2i}.
\end{align*}
\end{thmABC}
\noindent We illustrate Theorem~\ref{theorem B} with a well-known example.

\begin{example}\label{exa:heisenberg}
Let $G = U_3(\Fq)$ be the group of $3\times 3$ upper-unitriangular
matrices over~$\Fq$, where $q=p^f$.  Thus $|G|=q^3$, $a=2$ and
$b=1$. For odd~$p$, $G$ is isomorphic
to~$\exp(\mathfrak{f}_{2,2}(\Fq))$, where $\mathfrak{f}_{2,2}(\Fq)$ is
the $\Fq$-Lie algebra with $\Fq$-basis $(u,v,w)$, subject only to the
relations $[v,u]=w$, $[u,w]=[v,w]=0$.  With respect to this
$\Fq$-basis
$$A(\bfX)=\left(\begin{matrix}-X_2\\X_1\end{matrix}\right)\quad \text{
    and }\quad
  B(\bfY)=\left(\begin{matrix}&-Y_1\\Y_1&\end{matrix}\right).$$
  Theorem~\ref{theorem B} confirms the well-known formulae $\cc(G) =
  (q,q^2-1)_f$ and $\ch(G) = (q^2, q-1)_f$.  We note that $\fS(G)$ may
  be identified with $\{(u,v,w)\in\Fq^3 \mid wu=wv=0\}$, showing that
  $k(G)=q^2+q-1$.
\end{example}

In Section~\ref{five} we study generalizations of the relatively
free $p$-groups of exponent~$p$. For integers $r\geq 2$ and $c\geq 1$
we consider the free $\Fq$-Lie algebra $\frcq$ on $r$ generators and
nilpotency class~$c$, where $q=p^f$ is a power of a prime~$p>c$. The
Lazard correspondence associates the $p$-group $\Frcq = \exp(\frcq)$
to this $\Fq$-Lie algebra.  Our approach yields, for instance, a
simple, geometric proof of the following generalization
of~\cite[Theorem~5]{ItoMann/06} and of Example~\ref{exa:heisenberg}.

\begin{proposition} \label{ch Frc class 2} 
 Let $q=p^f$ be an odd prime power.  The character degrees of
 $F_{r,2}(\Fq)$ are $1,q,q^2,\dots,q^{\lfloor r/2 \rfloor}$.  For $0 \leq 2i
 \leq r$
$$\ch_{if}(F_{r,2}(\Fq)) =
 q^{r+i^2-3i}\frac{\prod_{j=0}^{2i-1}(q^{r-j}-1)}{\prod_{j=0}^{i-1}(q^{2(i-j)}-1)}.$$
\end{proposition}

\begin{proof}
  We fix an $\Fq$-basis $(x_1,\dots,x_r,y_{k\ell} \mid 1 \leq k < \ell
  \leq r)$ for $\mathfrak{f}_{2,r}(\Fq)$, subject to the relations
  $[x_\ell,x_k]=y_{k\ell}$ for $1\leq k < \ell \leq r$. Note that $a=r$
  and~$b=\binom{r}{2}$. The commutator matrix $B(\bfY)$ with respect
  to this basis is the generic skew-symmetric matrix in variables
  $Y_{k\ell}$ for $1\leq k<\ell \leq r$, so $B(\bfY)_{k\ell}=-Y_{k\ell}$.  It is
  well known that, for $0 \leq 2i\leq r$, the set
  $\{\bfy\in\Fq^{\binom{n}{2}} \mid \rk(B(\bfy))=2i\}$ has cardinality
  $$\nu_{if}(F_{r,2}(\Fq)) :=
  q^{i(i-1)}\frac{\prod_{j=0}^{2i-1}(q^{r-j}-1)}{\prod_{j=0}^{i-1}(q^{2(i-j)}-1)};$$
  see~\cite[Equation~(7.5)]{CarlitzHodges/56}. Theorem~\ref{theorem B}
  implies that $\ch_{if}(F_{r,2}(\Fq))=q^{r-2i}\nu_{if}(F_{r,2}(\Fq))$.
\end{proof}
 
Recall that the `Witt formula' is defined, for $i\in\N$, by
\begin{equation} \label{witt-function}
W_r(i):=\frac{1}{i}\sum_{d|i}\mu(d)r^{i/d},
\end{equation}
where $\mu$ denotes the M\"obius function; cf., for example,
\cite[Chapter~11]{Hall/76}.  We define
$$n(r,c):=\begin{cases}\sum_{i=1}^m W_r(i)&\text{ if
  }c=2m+1,\\\sum_{i=1}^{m-1}W_r(i)+ \lfloor
  \frac{W_r(m)}{2}\rfloor&\text{ if }c=2m.\end{cases}$$ 

\begin{theorem}\label{theorem chi Frc}
  Assume that $(r,c)\neq (2,3)$, that $p>c$ and let $q$ be a power
  of~$p$.  The character degrees of $\Frcq$ are
  $1,q,q^2,\dots,q^{n(r,c)}.$
\end{theorem}

\noindent The character vectors of the groups $F_{2,3}(\Fq)$ are given in
Proposition~\ref{prop:F23}. For $i\in[c]$ we define
$$k(r,c,i):=-\delta_{i<(c+1)/2}+\sum_{\ell=1}^{c-i}W_r(\ell).$$ 

\begin{theorem}\label{theorem conjugacy classes}
  Assume $p>c$ and let $q$ be a power of~$p$.  The class sizes of
  $\Frcq$ are $q^{k(r,c,i)}$ for $1 \leq i \leq c$. For $j\geq 1$
\begin{equation}\label{equ:k}
\cc_{jf}(\Frcq)=\sum_{\{i\in[c-1] \, \mid \,
k(r,c,i)=j\}} \left( q^{W_r(i)}-1 \right)q^{-j+\sum_{\ell=i+1}^{c}W_r(\ell)},
\end{equation}
and $\cc_0(\Frcq)=\vert Z(\Frcq) \vert = q^{W_r(c)}$.
\end{theorem}
\noindent Observe that the function $i\mapsto k(r,c,i)$ is injective unless
$r=2$ and~$c\in\{3,4\}$; in these cases the sum in~\eqref{equ:k} has at most two
nonzero summands. Generically it has at most one.

Theorems~\ref{theorem chi Frc} and \ref{theorem conjugacy classes}
will be proven in Section~\ref{five}.

\section{Proofs of Theorems~\ref{theorem A} and \ref{theorem B}}
\label{three}
The Lazard correspondence between $p$-groups and Lie rings of
nilpotency class $c < p$ allows us to linearize the problem of
enumerating conjugacy classes and characters.  Let $G$ be a finite
$p$-group of nilpotency class $c < p$, with associated Lie ring~$\mfg
= \log(G)$.

\subsection{Counting conjugacy classes}
It follows from straightforward calculations with the Hausdorff series
that $\log$ induces an order-preserving correspondence between
subgroups of $G$ and subalgebras of~$\mfg$, and $\log$ maps normal
subgroups to ideals. In particular, $|G/Z|=|\mfg/\mfz|$ and
$|G'|=|\mfg'|$, and centralizers in $G$ correspond to centralizers in
$\mfg$.  Thus
\begin{align*}
\cc_i(G) = & \#\{\text{conjugacy classes of $G$ of cardinality
  $p^i$}\} \nonumber \\ =& \#\{g\in G \mid |G:C_G(g)|=p^{i}\}
p^{-i}\nonumber \\ =&\#\{x\in\mfg \mid |\mfg:C_\mfg(x)|=p^{i}\}
p^{-i}\nonumber \\ =&\#\{x\in\mfg/\mfz \mid |\mfg/\mfz:C_{\mfg/\mfz}(x)| =
p^{i}\}\; |\mfz|\,p^{-i} .
\end{align*}
The last equality reflects the fact that the 
centralizer of an element only
depends on its coset modulo the centre. 
For $x\in\mfg/\mfz$ we define
\begin{alignat*}{2}
 \ad_x & :\mfg/\mfz \rightarrow\mfg', &\qquad  z &\mapsto [z,x] \\
 \ad^\star_x & :\widehat{\mfg'}\rightarrow \widehat{\mfg/\mfz}, &\qquad 
\omega &\mapsto \omega \circ  \ad_x. 
\end{alignat*}
Hence 
\begin{equation}\label{equ:cci}
\cc_i(G) = \#\{x\in\mfg/\mfz \mid |\im(\ad_x)|=p^i\} \; |\mfz|\,p^{-i} = 
\#\{x\in\mfg/\mfz \mid |\ker(\ad^\star_x)|=p^{-i}|\widehat{\mfg'}|\}\; |\mfz|\,p^{-i}.
\end{equation}

\subsection{Kirillov's orbit method and counting characters}
The Kirillov orbit method offers a linearization of the character
theory of $G$ in terms of co-adjoint orbits: characters of $G$
correspond to orbits in
$\widehat{\mfg}:=\text{Hom}_\Z(\mfg,\C^\times)$, the Pontryagin dual
of~$\mfg$, under the co-adjoint action $\Ad^\star$ of $G$
on~$\widehat{\mfg}$. The following is well-known; see, for example,
\cite[Theorem~2.6]{BoyarchenkoSabitova/08}
or~\cite[Theorem~4.4]{Gonzalez/09}.

\begin{theorem}\label{theorem kirillov}
  Let $G=\exp(\mfg)$ be a finite $p$-group of nilpotency
  class~$c<p$. Let $\Omega\subseteq\widehat{\mfg}$ be a co-adjoint
  orbit and~$\omega\in\Omega$.
\begin{enumerate}
\item There exists a polarizing subalgebra $\mfh\subseteq\mfg$ for the
  bi-additive, skew-symmetric form
  $B_{\omega}:\mfg\times\mfg\rightarrow\C^\times, (u,v)\mapsto
  \omega([u,v])$: namely, a subalgebra $\mfh$ that is maximal with
  respect to the property that
 ~$B_\omega\left|_{\mfh\times\mfh}\right.\equiv1$. Setting
$$\Rad(B_{\omega}):=\{u\in\mfg \mid 
B_\omega(u,v)=1 \emph{ for all }v\in\mfg\},$$ 
$\exp(\Rad(B_{\omega}))$ is the $\Ad^\star$-stabilizer $\Stab_G(\omega)$,
and $|\mfg:\mfh|=|\mfh:\Rad(B_{\omega})|$. Thus, with
$H:=\exp(\mfh)$, 
\begin{equation*}
|\Omega|^{1/2}=|G:\Stab_G(\omega)|^{1/2}=|\mfg:\Rad(B_{\omega})|^{1/2}=|\mfg:\mfh|=|G:H|.
\end{equation*}
\item Viewing $\omega$ as a function on $G$ $($via $\log)$, the
  function $\omega|_H$ is a one-dimensional representation of $H$. The
  induced representation $U_\Omega:=\Ind_H^G\omega$ of $G$ is
  irreducible, independent of~$\omega$, and has dimension
  $|\Omega|^{1/2}$. All irreducible complex representations of $G$
  have this form.
\item The character of $U_\Omega$ is given by
  $|\Omega|^{-1/2}\sum_{\omega\in\Omega}\omega(g)$, for $g\in G$.
\end{enumerate}
\end{theorem}

\begin{remark}
A Kirillov orbit method for torsion-free finitely generated nilpotent
pro-$p$ groups of class $2$ that holds for all primes $p$ is presented
in \cite[Section~2.4]{StasinskiVoll/11}.  We expect that it can be
used to prove the conclusions of Theorem~\ref{theorem kirillov} for
$2$-groups of class~$2$.
\end{remark}

Theorem~\ref{theorem kirillov} reduces the problem of enumerating the
characters of $G$ to that of computing the indices in $\mfg$ of the
radicals $\Rad(B_{\omega})$, as $\omega$ ranges over $\widehat{\mfg}$.
In fact, given $\omega\in\widehat{\mfg}$, the form $B_\omega$ only
depends on the restriction of $\omega$ to $\mfg'$. Given
$\omega\in\widehat{\mfg'}$ we therefore write $B_\omega$ for
$B_{\wt{\omega}}$, where $\wt{\omega}\in\widehat{\mfg}$ is any
extension of $\omega$. With this notation, Theorem~\ref{theorem
  kirillov} implies that
\begin{align}
  \ch_i(G)=& \#\{\text{irreducible complex characters of $G$ of degree
    $p^i$}\}\nonumber\\ = &\#\{\text{co-adjoint orbits
  }\Omega\subseteq \widehat{\mfg} \text{ of size
    $p^{2i}$}\}\nonumber\\ =&\#\{\omega\in \widehat{\mfg} \mid
  |\mfg:\Rad(B_{\omega})|=p^{2i}\} \,
  p^{-2i}\nonumber\\ =&\#\{\omega\in \widehat{\mfg'} \mid
  |\mfg:\Rad(B_{\omega})| = p^{2i}\} \, |\mfg/\mfg'| \, p^{-2i}
  \nonumber\\ =&\#\{\omega\in \widehat{\mfg'} \mid
  |\Rad(B_{\omega})/\mfz| = p^{-2i} | \mfg/\mfz | \} \, |\mfg/\mfg'|
  \, p^{-2i}.
\label{equ:chi}
\end{align}

\subsection{Proof of Theorem~\ref{theorem A}}
For $i\in\N_0$ we define
\begin{align*}
  \mu_i(G)&:=\#\{x\in\mfg/\mfz \mid
  |\ker(\ad^\star_x)|=p^{-i}|\widehat{\mfg'}|\},\\
  \nu_i(G)&:=\#\{\omega\in\widehat{\mfg'} \mid
  |\Rad(B_{\omega})/\mfz|=p^{-2i}|\mfg/\mfz|\}.
\end{align*}
Equations \eqref{equ:cci} and \eqref{equ:chi} imply that
$\cc_i(G)=\mu_i(G)|\mfz|p^{-i}$ and
$\ch_i(G)=\nu_i(G)|\mfg/\mfg'|p^{-2i}$.  For $x\in\mfg/\mfz$ and
$\omega\in\widehat{\mfg'}$, observe that $x\in\Rad(B_{\omega})/\mfz
\text{ if and only if }\omega\in\ker(\ad^\star_x)$.  Thus
\begin{align} 
 \fS(G)=& \{(x,\omega)\in\mfg/\mfz\times\widehat{\mfg'} \mid
 \omega([x,z])=1 \text{ for all } z \in\mfg/\mfz
 \}\nonumber\\=&\{(x,\omega)\in\mfg/\mfz\times\widehat{\mfg'} \mid
 \omega\in\ker(\ad^\star_x)\}\nonumber\\=&\{(x,\omega)\in\mfg/\mfz\times\widehat{\mfg'}
 \mid x\in\Rad(B_{\omega})/\mfz\}\label{S radical}.
\end{align}
Using the natural projections $\pi_1:\fS(G)\rightarrow\mfg/\mfz$ and
$\pi_2:\fS(G)\rightarrow\widehat{\mfg'}$, we see that
\begin{align*}
\mu_i(G)&=\#\{x\in\mfg/\mfz \mid |\pi_1^{-1}(x)|=p^{-i}|\widehat{\mfg'}|\},\\
\nu_i(G)&=\#\{\omega\in\widehat{\mfg'} \mid |\pi_2^{-1}(\omega)|=p^{-2i}|\mfg/\mfz|\}.
\end{align*}
We obtain two descriptions of the class number $k(G)$:
$$|\mfg'|k(G) =
|\mfg'|\sum_i\cc_i(G)=|\mfz|\sum_ip^{-i}|\mfg'|\mu_i(G)=|\mfz||\fS(G)|,$$
$$ |\mfg/\mfz|k(G) =
|\mfg/\mfz|\sum_i\ch_i(G)=|\mfg/\mfg'|\sum_ip^{-2i}|\mfg/\mfz|\nu_i(G)=|\mfg/\mfg'||\fS(G)|.$$
We deduce that $k(G) =|\fS(G)| \;|\mfz|\,|\mfg'|^{-1} = |\fS(G)| \;
|Z(G)| \, |G'|^{-1}.$ This proves Theorem~\ref{theorem A}.

\subsection{Proof of Theorem~\ref{theorem B}}
Recall that $\lri$ is a compact discrete valuation ring with residue
field $\kk = \lri/\mfp$ of characteristic $p$, and that $\mfg$ is a
finite, nilpotent $\lri$-Lie algebra of class $c<p$ with the property
that $\mfg/\mfz$ and $\mfg'$ are annihilated by~$\mfp$. Further recall
the isomorphisms $\phi:\mfg/\mfz \rarr \kk^a$ and
$\psi:\wh{\mfg'}\rarr \kk^b$ introduced in Section~\ref{two}. Consider
the $p$-group $G=\exp(\mfg)$. By~\eqref{S radical},
$$\fS(G) = \{(x,\omega)\in\mfg/\mfz\times\widehat{\mfg'} \mid
x\in\Rad(B_{\omega})/\mfz\}.$$ The following lemma, proved analogously
to \cite[Lemma~3.3]{AKOVI/10}, characterizes membership of $\fS(G)$ in
terms of the above coordinate systems for $\mfg/\mfz$
and~$\widehat{\mfg'}$.

\begin{lemma}\label{lemma 1}
  Let $x\in \mfg/\mfz$ and $\omega \in \widehat{\mfg'}$ correspond to
  $\phi(x)=\bfx \in \kk^a$ and $\psi(\omega)=\bfy \in\kk^b$.  Then
\begin{equation*}\label{lemma matrices}
  x\in\Rad(B_\omega)/\mfz \text{ if and only if } 
  A(\bfx)\bfy^{\trans} = B(\bfy)\bfx^{\trans} = 0.
\end{equation*}
\end{lemma}

Now let $\Lri$ be a finite, unramified extension of~$\lri$, with
residue field isomorphic to $\Fq=\Fpf$, say. Applying
Theorem~\ref{theorem A}, \eqref{S radical} and Lemma~\ref{lemma 1} to
$\mfg(\Lri)=\mfg\otimes_\lri\Lri$ reduces the computation of the
class and character vector of $G(\Lri)=\exp(\mfg(\Lri))$ to the
problem of counting the solutions to linear equations over $\Fq$.  In
particular, all class sizes and character degrees are powers
of~$q=p^f$. For $i \geq 0$,
\begin{align*}
\mu_{if}(G(\Lri))&= \#\left\{\bfx\in\Fq^a \mid \rk(A(\bfx)) =
i\right\},
\\ \nu_{if}(G(\Lri))&
= \#\left\{ \bfy \in \Fq^b \mid
\rk(B(\bfy)) = {2i} \right\}.
\end{align*}
This proves Theorem~\ref{theorem B}.

\section{Consequences of Theorems~\ref{theorem A} and \ref{theorem B}}
\label{four}
\subsection{Isoclinism}
  Recall from \cite{Hall/40} that two $p$-groups $G_1$ and $G_2$ are
  \emph{isoclinic} if there are isomorphisms $\theta: G_1/Z_1
  \rightarrow G_2/Z_2$ and $\phi:G_1'\rightarrow G_2'$ such that, for
  all $\alpha,\beta\in G_1'$, $\phi([\alpha,\beta])=[\theta(\alpha
    Z_1),\theta(\beta Z_1)]$.  The pair $(\theta,\phi)$ is an
  \emph{isoclinism} between $G_1$ and~$G_2$.

  If $G_1$ and $G_2$ have nilpotency class less than $p$ and
  $(\theta,\phi)$ is an isoclinism between $G_1 = \exp(\mfg_1)$ and
  $G_2 = \exp(\mfg_2)$, then there is a pair of associated maps
  $(\Theta,\Phi^{-1})$, where $\Theta = \log\, \circ\, \theta \circ
  \exp$, $\Phi^{-1}=\log\, \circ\, \phi^{-1}\circ\exp$ and
  $\widehat{\Phi^{-1}}: \widehat{\mfg_1}\rightarrow \widehat{\mfg_2}$,
  $\omega_1\mapsto \omega_1\circ \Phi^{-1}$.  The isoclinism
  $(\theta,\phi)$ induces a bijection
  $(\Theta,\widehat{\Phi^{-1}}):\fS(G_1)\rightarrow \fS(G_2)$, where,
  for $i\in\{1,2\}$, $\fS(G_i)$ are as defined in Theorem~\ref{theorem
    A}; in particular
$$\fS(G_1)=\{(x_1,\omega_1)\in (\mfg_1/Z(\mfg_1))\times \widehat{\mfg_1'} \mid
\omega_1([x_1,z_1])=1\text{ for all }z_1\in\mfg_1/Z(\mfg_1)\}.$$ 

By definition, $\omega_1([x_1,z_1]) =1$ if and only if
$\omega_1(\Phi^{-1}([\Theta(x_1),\Theta(z_1)]))=1$; this holds if and
only if $\widehat{\Phi^{-1}}(\omega_1)([\Theta(x_1),\Theta(z_1)])=1$.
Therefore $(\Theta,\widehat{\Phi^{-1}})(\fS(G_1))=\fS(G_2)$. This, of
course, merely reflects the well-known fact that isoclinic groups
have, up to multiplication by $p$-powers, identical class (and
character) vectors.

\subsection{Pfaffian hypersurfaces}
Boston and Isaacs \cite{BostonIsaacs/04} studied the class vectors of
some $p$-groups of class $2$ and exponent~$p$.  In this section we
prove a generalization and extension of~\cite[Theorem
  3.2]{BostonIsaacs/04}. We first describe our broader context.  Let
$\lri$ be a compact, discrete valuation ring with residue field
$\lri/\mfp$, which we identify with~$\Fq$, where $q=p^f$ is an odd
prime power. Let $\mfg$ be a finite, nilpotent $\lri$-Lie algebra of
class~$2$.  Assume that $\mfg/\mfz$ and $\mfg'$ are annihilated by
$\mfp$, so that Theorem~\ref{theorem B} applies.  The coordinate
systems introduced in Section~\ref{two} identify ${\mfg/\mfz}$ with
$\Fq^a$ and $\mfg'$ with~$\Fq^b$, where we write $a$ for the
$\lri$-rank of $\mfg/\mfz$ and $b$ for the $\lri$-rank of $\mfg'$.
Recall from Definition \ref{def:commutator matrices} the commutator
matrix $B$ associated to $\mfg$ with respect to the chosen bases. We
denote by $\pP^{b-1}(\Fq)$ the $(b-1)$-dimensional projective space
over~$\Fq$. Note that $\rk(B(\wt{\bfy}))$ is well-defined
for~$\wt{\bfy}=(\wt{y_1}:\dots:\wt{y_b})\in\pP^{b-1}(\Fq)$. We write
$G=\exp(\mfg)$ and recall that, by Theorem~\ref{theorem B}, $\cs(G)$
and $\ch(G)$ consist of powers of~$q=p^f$.

\begin{theorem} \label{proposition pfaffian} Assume that $a>2$,
\begin{equation}\label{equation BI condition}
 \left\{{\rm rk}(B(\wt{\bfy})) \mid \wt{\bfy}\in
 \pP^{b-1}(\Fq)\right\}=\{a-2,a\}
\end{equation}
and that, for every line $L\subset\pP^{b-1}(\Fq)$, there exists
$\wt{\bfy}\in L$ such that ${\rm rk}(B(\wt{\bfy}))=a$. Let
$$n:=\#\{\wt{\bfy}\in\pP^{b-1}(\Fq) \mid {\rm rk}(B(\wt{\bfy}))=a-2\}.$$ Then

\begin{equation} \label{formula BI classes}
  \cc_{if}(G)=\begin{cases}|Z| & \text{ if } i=0, \\
    |Z|q^{-b+1}n(q^2-1) & \text{ if } i=b-1, \\|Z|q^{-b}(q^a-1-n(q^2-1)) & \text{ if }i=b,\\
    0 & \text{ otherwise},
\end{cases}
\end{equation} 
\begin{equation} \label{formula BI characters}
  \ch_{if}(G)=\begin{cases}|G/G'| & \text{ if } i=0, \\
    |G/G'|q^{-a+2}n(q-1) & \text{ if } i=a/2-1, \\|G/G'|q^{-a}(q^b-1-n(q-1)) & \text{ if }i=a/2,\\
    0 & \text{ otherwise}.
\end{cases}
\end{equation}
In particular
$$k(G)=|G|(q^{-a}+q^{-b} + q^{-a-b}(n(q^2-1)(q-1)-1)).$$
\end{theorem}

\begin{remark}\label{rem:geometry}
  Geometrically, the hypotheses of Theorem~\ref{proposition pfaffian}
  imply that the projective Pfaffian hypersurface defined by $\Pf(B)$
  contains no lines over~$\Fq$. In particular, the Pfaffian is not
  identically zero, and thus $a$ is even;
  cf.\ Remark~\ref{rem:pfaffian}. Hypothesis~\eqref{equation BI
    condition} implies that $b>1$ and is satisfied if (but not only
  if) the Pfaffian defines a smooth hypersurface in $\pP^{b-1}(\F_q)$;
  cf.~\cite[Lemma~5]{Voll/09}.
\end{remark}

\begin{proof}
  Remark~\ref{rem:geometry} shows that $a$ is even. To prove
  \eqref{formula BI characters} we observe that, by the hypotheses,
$$\nu_{if}(G)=\begin{cases} 1 &\text{ if }i=0\\n(q-1)& \text{ if }
  i=a/2-1,\\ q^b-1-n(q-1)& \text{ if }i=a/2,\\0 &\text{
    otherwise.}\end{cases}$$ The claim about $\ch(G)$ then follows
from Theorem~\ref{theorem B} which asserts that
$\ch_{if}(G)=\nu_{if}(G)\,|G/G'|q^{-2i}$ for $i\in[a/2]_0$.

To prove \eqref{formula BI classes} it suffices to show that, firstly,
$\mu_{if}(G)=0$ for $i\in[b-2]$ and, secondly,
$\mu_{(b-1)f}(G)=n(q^2-1)$. Indeed, clearly $\mu_0(G)=1$ and
$\sum_{i=0}^{b}\mu_{if}(G)=q^a$, so that
$\mu_{bf}(G)=q^a-1-n(q^2-1)$. The claim about $\cc(G)$ then follows
from Theorem~\ref{theorem B} which asserts that $\cc_{if}(G) =
\mu_{if}(G)\,|Z|q^{-i}$ for~$i\in[b]_0$.

Given $\bfy\in\Fq^b$ we view $B(\bfy)$ as the matrix of an
endomorphism of $\Fq^a$, whose kernel we denote by
$\ker(B(\bfy))$. Likewise, given $\bfx\in\Fq^a$, we view $A(\bfx)$ as
the matrix of the linear map $\Fq^b \rarr \Fq^a, \bfy \mapsto \bfy
A(\bfx)^{\trans}$, whose kernel we denote by~$\ker(A(\bfx))$.

Let $\bfy\in\Fq^b$ be one of the $n(q-1)$ elements with
$\rk(B(\bfy))=a-2$, so~$\dim(\ker(B(\bfy)))=2$.  Observe that
$\rk(A(\bfx))<b$ for all $\bfx\in\ker(B(\bfy))$. We claim that
$\rk(A(\bfx))=b-1$ for all such $\bfx$ which are nonzero. Indeed,
assume that $\bfx\neq{\bf 0}$ with $\rk(A(\bfx))\leq b-2$. Let $V\leq
\Fq^b$ be a $2$-dimensional subspace of $\ker(A(\bfx))$.  For every
$\bfy\in V$ we deduce using Lemma \ref{lemma 1} that~$\bfy
A(\bfx)^{\trans} = \bfx B(\bfy)^{\trans}=0$.  Therefore $V$ defines a
line in $\pP^{b-1}(\Fq)$ on which no point $\wt{\bfy}$ satisfies
$\rk(B(\wt{\bfy}))=a$, contradicting our hypotheses. Thus
$\rk(A(\bfx))=b-1$. This shows that $\mu_{if}(G)=0$ for~$i\in[b-2]$,
establishing the first claim.

Every $\bfy \in \Fq^b\setminus\{{\bf 0}\}$ such that
$\rk(B(\bfy))=a-2$ gives rise to $q^2-1$ elements $\bfx \in
\Fq^a\setminus\{{\bf0}\}$ such that $\rk(A(\bfx))=b-1$, namely the
nonzero elements of the $2$-dimensional
space~$\ker(B(\bfy))$. Likewise, every
$\bfx\in\Fq^a\setminus\{{\bf0}\}$ such that $\rk(A(\bfx))=b-1$ gives
rise to $q-1$ elements $\bfy$ with this property, namely the nonzero
elements of its nullspace. Thus
$\mu_{(b-1)f}(G)=n(q-1)\frac{(q^2-1)}{(q-1)}=n(q^2-1)$, establishing
the second claim.
\end{proof}

\begin{example}
  Let $p$ be a prime and $\alpha \in \F_p^\times$. Let $\mfg_\alpha$
  be the 9-dimensional nilpotent $\Fp$-Lie algebra of class 2 with
  $\F_p$-basis $(e_1, \ldots, e_6, f_1, f_2, f_3)$ subject only to the
  relations $[e_1, e_4] = f_1$, $[e_1, e_5] = f_2$, $[e_1, e_6] =
  \alpha f_3$, $ [e_2, e_4] = f_3$, $[e_2, e_5] = f_1$, $[e_2, e_6] =
  f_2$, $[e_3, e_4] = f_3$, $[e_3, e_6] = f_1$, where $0 \not= \alpha
  \in \F_p$.  With respect to this basis, the commutator matrices are:
$$A(\bfX) = \left(\begin{matrix}
    -X_4 & -\alpha X_6 & -X_5 \\
    -X_5 & -X_4       & -X_6 \\
    -X_6 & -X_4       & 0 \\
    X_1 &  X_2 + X_3       & 0 \\
    X_2 &  0               & X_1 \\
    X_3 & \alpha X_1 & X_2
\end{matrix}\right);
$$
$$ B(\bfY) = \left(\begin{matrix}
 0 & U(\bfY)  \\
-U(\bfY)^{\trans} & 0
\end{matrix}\right),
\text{ where } U(\bfY) = \left(\begin{matrix}
    Y_{1} & Y_{2} & \alpha Y_{3}\\
    Y_{3} & Y_{1} &  Y_{2}\\
    Y_{3} & 0 & Y_1 \\
\end{matrix}\right).
$$
Boston and Isaacs \cite{BostonIsaacs/04} study the groups $G_\alpha =
\exp(\mfg_\alpha)$, which satisfy the hypotheses of Theorem
\ref{proposition pfaffian} if $p$ is odd.  They prove that
$k(G_\alpha) = p^6+p^3-1+n_\alpha (p^2-1)(p-1)$, where $n_\alpha =
\#\{\wt{\bfy}\in\pP^{2}(\Fp) \mid {\rm rk}(B(\wt{\bfy}))=4\}$, which
accords with Theorem~\ref{proposition pfaffian}. They also show that
$\#\{n_\alpha \mid \alpha \in \F_p^\times \} \rightarrow \infty$ as
$p\rightarrow \infty$. Thus they establish that the number of
different values assumed by $k(G)$ as $G$ runs over all groups of
order $p^{9}$ tends to infinity with $p$.
\end{example}

\begin{example}
 Let $\mfg$ be the 8-dimensional nilpotent $\Fq$-Lie algebra of class
 2 with $\Fq$-basis $(e_1,\ldots, e_4, f_1, \ldots, f_4)$ subject only
 to the relations $[e_1,e_3] = f_1, [e_1, e_4] = f_2, [e_2, e_3] =
 f_3, [e_2, e_4] = f_4$.  The class and character vectors of $G =
 \exp(\mfg)$ are the following:
 \begin{align*}(\cc_{if}(G))_{i\in\{0,1,2,3\}}&=(q^4,0,2(q^2-1)q^2,q(q^2-1)^2)_f,\\
   (\ch_{if}(G))_{i\in\{0,1,2\}}&
   =(q^4,q^2(q-1)(q+1)^2,q^4-1-(q+1)^2(q-1))_f.
\end{align*}
This follows from inspection of the commutator matrices
$$A(\bfX)=\left(\begin{matrix}X_{3}&X_{4}&&\\&&X_{3}&X_{4}\\-X_{1}&&-X_{2}&\\&-X_{1}&&-X_{2}\end{matrix}\right),\quad
B(\bfY)=\left(\begin{matrix}&&Y_{1}&Y_{2}\\&&Y_{3}&Y_{4}\\-Y_{1}&-Y_{3}&&\\-Y_{2}&-Y_{4}\end{matrix}\right).$$
The class vector differs from~\eqref{formula BI classes}, but the
character vector agrees with~\eqref{formula BI characters}.  The
hypothesis of Theorem~\ref{proposition pfaffian} regarding lines in
$\pP^3(\Fq)$ is not satisfied. We observe that the factor $2(q+1)$ of
$\cc_{2f}(G)$ is the number of lines on the Pfaffian hypersurface, the
quadric surface defined by $Y_{1}Y_{4}-Y_{2}Y_{3}=0$.
\end{example}

\subsection{Prescribing class sizes and character degrees} 
It is known that every finite set of $p$-powers containing $1$ can be
realized as the class sizes or character degrees of a finite
$p$-group; cf.\ \cite{CosseyHawkes/00} and \cite{Isaacs/86}
respectively. Such results can be obtained readily using Theorems
\ref{theorem A} and \ref{theorem B}. Throughout this section let $p$
be an odd prime.

\subsubsection{}We show how to obtain the result of \cite{Isaacs/86}.
Let $I\subset\N$ be finite and let~$j = \max(I)$.  To construct a
$p$-group $G$ such that $\cd(G)=\{p^i \mid i\in I_0\}$, consider the
$\Fp$-Lie algebra~$\mfg$, with $\Fp$-basis consisting of
$x_1,\dots,x_{2j}$ and $y_i$ for $i\in I$, subject only to the
relations
$$ [x_r, x_t] = \begin{cases} y_i &\text{ if }t-r=i\in I,\\0 &\text{
    otherwise, } \end{cases}\quad \text{ for $r,t\in[2j]$.}
$$

\noindent The commutator matrix $B(\bfY)\in\Mat(2j,\Fp[\bfY])$ in
variables $Y_i$ for $i \in I$ with respect to this basis is the sum of
the $(2j\times 2j)$-matrices
$$\left(\begin{matrix}0&Y_i
    \Id_{i}&\\-Y_i \Id_{i}&0&\\&&0_{2(j-i)}\end{matrix}\right), \text{
  where $i\in I$.}$$ Clearly $\{ \rk(B(\bfy)) \mid \bfy \in
\Fp^{|I|}\} = I_0$.  Theorem~\ref{theorem B} implies that
$\cd(\exp(\mfg)) = \{p^i \mid i\in I_0\}$.

\subsubsection{}
Fern\'andez-Alcober and Moret\'o \cite{FernandezMoreto/01} prove that
for every two integers $u,v > 1$ there exists a finite $p$-group $H$ %
of class 2 such that $|\cd(H)|=u$ and $|\cs(H)|=v$.  As part of their
proof, they construct, for given $l, n \in \N$, a $p$-group $G$ with
$\cd(G)=\{1, p^l\}$, and $\cs(G)=\{1, p, \ldots, p^l, p^n\}$;
cf.~\cite[Lemma~2.2]{FernandezMoreto/01}.

We show how to construct such a group~$G$.  Consider the $\Fp$-Lie
algebra~$\mfg$, with $\Fp$-basis $(x_1,\dots,x_{l}, \wt{x}_1, \ldots,
\wt{x}_{l + n - 1}, y_1 \ldots, y_n)$, subject only to the relations:
\begin{equation*}
[x_i, \wt{x}_j] = \begin{cases}
y_{j -i +1} &\text{ if } i \leq j \leq i + n - 1,\\
0 &\text{ otherwise, }\end{cases}\quad\text{ for $i\in[l], j \in [l + n - 1]$} 
\end{equation*}
With respect to this basis the commutator matrix $B(\bfY)\in\Mat(2l +
n - 1, \Fp[\bfY])$ is
$$\left(\begin{matrix} 0 & U(\bfY) \\ -U(\bfY)^{\trans} & 0
\end{matrix}\right),
\text{ where } U(\bfY) = \left(\begin{matrix} Y_{1} & Y_{2} & \cdots & Y_n &
  & \\ & Y_{1} & Y_2 & \cdots & Y_n & \\ & & \ddots & \ddots & &
  \ddots & \\ & & & Y_1 & Y_2 & \cdots & Y_n
\end{matrix}\right).
$$
The commutator matrix $A(\bfX)\in\Mat((2l + n - 1) \times n,
\Fp[\bfX])$ is 
$$\left(\begin{matrix}
X_{l+1} & X_{l + 2} & \cdots & X_{l + n} \\
X_{l+2} & X_{l + 3} & \cdots & X_{l + n + 1} \\
\vdots & & & \vdots \\
X_{2l} & X_{2l + 1} & \cdots & X_{2l + n - 1} \\
-X_{1} & & &  \\
\vdots & -X_{1} & &  \\
-X_{l} & \vdots & \ddots  &  \\
& -X_{l}  & & -X_1 \\
& & \ddots &  \vdots \\
 & & & -X_l 
\end{matrix}\right).
$$ Note that $\{\rk (B(\bfy)) \mid \bfy \in \F_p^n\} = \{0, l\}$ and
$\{\rk(A(\bfx)) \mid \bfx \in \F_p^{2l + n - 1}\} = \{0,1, \ldots, l,
n\}$.  Therefore $\cd(\exp(\mfg))$ and $\cs(\exp(\mfg))$ are as
stated.

\section{Relatively free $p$-groups of exponent $p$}\label{five}
Let $p$ be a prime and let $r\geq 2$, $c\geq1$.  Ito and
Mann~\cite{ItoMann/06} study the numbers of classes and characters of
the relatively free $p$-groups in the variety of groups of exponent
$p$ and nilpotency class $c$ on $r$ generators.  Our methods apply
when $c < p$. We consider, more generally, the groups $\Frcq :=
\exp(\frcq)$, where $q=p^f$ and $\frcq$ is the free nilpotent
$\Fq$-Lie algebra of class $c$ on $r$ generators.  The orders of the
terms of the lower central series of $\frcq$, and so of~$\Frcq$, may
be expressed in terms of the Witt formula (\ref{witt-function}): for
$i\in[c]$,
\begin{equation}\label{equation witt function}
|\gamma_i(\frcq):\gamma_{i+1}(\frcq)| = |\gamma_i(\Frcq):\gamma_{i+1}(\Frcq)|=q^{W_r(i)};
\end{equation}
cf.\ \cite[Proposition 1]{ItoMann/06}.  We often write $F$ for
$\Frcq$, $\mff$ for $\frcq$ and $W$ for $W_r$.

\subsection{Proof of Theorem~$\ref{theorem conjugacy classes}$} 
\label{subsec:classes}
We first prove a lemma about free Lie algebras over arbitrary fields.
Let $\fieldK$ be a field and $L$ a free $\fieldK$-Lie algebra of rank
at least~2. We fix a Lie basis $\mathcal{B}$ for $L$ and, for
$m\in\N$, define $L_m$ as the $\fieldK$-linear span of $m$-fold Lie
products of elements of $\mathcal{B}$.  The standard grading $L =
\oplus_{m=1}^\infty L_m$ determines the lower central series
filtration $L^i := \oplus_{j =i}^\infty L_j$ of $L$. By convention,
$L^i:=L$ for $i\leq 0$.

Let $u=\sum_{i=1}^\infty u_i\in L$ be a nonzero element in the
standard grading, i.e.\ $u_i\in L_i$ for all~$i$. We denote by
$\underline{u}$ the nonzero homogeneous component of lowest degree of
$u$ in the standard grading: namely, $\underline{u}:=u_{\deg(u)}$,
where $\deg(u) :=\min\{i\in\N \mid u_i\neq 0\}$ is the degree of~$u$.

\begin{lemma} 
  Let $u,v\in L$ with $u\neq 0$.  If $[u, v] \in L^{i+1}$ for some
  positive integer $i$, then $v \in \fieldK u + L^{i + 1 -
    \deg{{u}}}$.
\end{lemma}

\begin{proof}
  Without loss of generality, assume~$v \not\in \fieldK u$.  By the
  Shirshov-Witt theorem \cite[Theorem 2.5]{Reutenauer/93}, every
  subalgebra of $L$ is free, and so $[u, v] \not=0$.  Now $[u, v] =
  [\underline{u}, \underline{v}] + z $ for some $z\in L$ with $\deg z
  > \deg u + \deg v$.

  If $[\underline{u}, \underline{v}] \not=0$ then it is homogeneous of
  degree $\deg u + \deg v$.  In this case $\underline{[u, v]} =
  [\underline{u}, \underline{v}]$ and so $[u, v]\not\in L^{\deg{{u}} +
    \deg{{v}} + 1}$.  By hypothesis $[u, v] \in L^{i+1}$, so $i + 1 <
  \deg{{u}} + \deg{{v}} + 1$ or, equivalently, $\deg{v} \geq i + 1 -
  \deg{{u}}$ which implies that $v \in L^{i + 1 - \deg{{u}}}$.

  If $[\underline{u}, \underline{v}]=0$ then the Shirshov-Witt theorem
  implies that $\underline{v} = k \underline{u}$ for some nonzero~$k
  \in \fieldK$.  Thus $\deg(v - k\underline{u}) > \deg u$.  Since
  $\underline{v - ku} \not\in \fieldK\underline{u}$ and so
  $[\underline{u}, \underline{v-kv}]\neq0$, we may apply the argument
  of the previous paragraph to $\underline{v - ku}$ instead of
  $\underline{v}$, deducing that $v-ku\in L^{i + 1 - \deg{{u}}}$.
\end{proof}
\noindent The Lazard correspondence implies the following.
\begin{lemma}\label{lemma 8}
  Let $g\in F=\Frcq$ and $i\in[c]$. If
  $g\in\gamma_i(F)\setminus\gamma_{i+1}(F)$ then
  $C_{F}(g)=\langle g,\gamma_{c-i+1}(F)\rangle$.
\end{lemma}

We now prove Theorem~$\ref{theorem conjugacy classes}$.  By
Lemma~\ref{lemma 8}, the conjugacy class sizes in $F=\Frcq$ are
the indices of the subgroups $\langle g,\gamma_{c-i+1}(F)\rangle$, for
$g\in\gamma_i(F)\setminus\gamma_{i+1}(F)$ and $i\in[c]$. If $i\geq
c-i+1$ then $\langle g,\gamma_{c-i+1}(F)\rangle = \gamma_{c-i+1}(F)$,
which has index $q^{\sum_{j=1}^{c-i}W(j)}$ in~$F$;
see~\eqref{equation witt function}. If $i<c-i+1$ then $|\langle
g,\gamma_{c-i+1}(F)\rangle:\gamma_{c-i+1}(F)|=q$, and so $|F:\langle
g,\gamma_{c-i+1}(F)\rangle|=q^{-1+\sum_{j=1}^{c-i}W(j)}$. Thus if
$g\in\gamma_i(F)\setminus \gamma_{i+1}(F)$ then the conjugacy class of
$g$ has size~$q^{k(r,c,i)}$.

The statement that $\cc_0(F) = q^{W(c)}$ follows immediately
from~\eqref{equation witt function}, as $Z(F)=\gamma_c(F)$. Note that
$0=k(r,c,c)$. To determine $\cc_{jf}(F)$ where $j=k(r,c,i)\neq 0$, it
suffices to count the elements in each
$\gamma_i(F)\setminus\gamma_{i+1}(F)$ such that $k(r,c,i)=j$ and to
observe that these elements fall into conjugacy classes of equal
size~$q^j$. Thus
\begin{align*}
 \cc_{jf}(F)&=\sum_{\{i\in[c-1] \,\mid\, k(r,c,i)=j\}}
|\gamma_i(F)\setminus\gamma_{i+1}(F)|\,q^{-j}\\ 
&=\sum_{\{i\in[c-1] \, \mid \,
k(r,c,i)=j\}} \left( q^{W_r(i)}-1 \right)q^{-j+\sum_{\ell=i+1}^{c}W_r(\ell)}.
\end{align*}
This concludes the proof of Theorem $\ref{theorem conjugacy classes}$.

\begin{corollary} \label{cor} Let $r\geq 2$, $c\geq 1$ and $q$ a power
  of $p>c$. The entries of the class vector $\cc(\Frcq)$, and hence
  also the class number~$k(\Frcq)$, are given by a polynomial in~$q$
  which depends only on $r$ and~$c$. Expanded in $q-1$, this
  polynomial has nonnegative coefficients.
\end{corollary}

\begin{proof}
 Theorem~\ref{theorem conjugacy classes} shows that the relevant
 quantities may be written as sums of terms of the form
 $q^\alpha$ and $(q^\beta-1)q^\gamma$
for nonnegative integers~$\alpha, \beta,\gamma$. 
\end{proof}

\begin{remark} 
  Corollary \ref{cor} may be compared to an analogous conjecture about
  the class vectors of the groups $U_n(\Fq)$ of upper-unitriangular
  matrices over~$\F_q$; cf.~\cite{VeralopezArregi/03}. Isaacs
  \cite{Isaacs/07} formulates a similar conjecture for the characters
  vectors of these groups; it is proved in~\cite{Evseev/11} for~$n\leq
  13$. In Remark~\ref{rem:negative coeff} we note that the
  corresponding statement for the character vectors of the groups
  $F_{2,5}(\Fq)$ is false.
\end{remark}

\subsection{Proof of Theorem~$\ref{theorem chi Frc}$}
We recall the well-known definition of a Hall basis; cf.~\cite{Hall/50}.
\begin{definition}\label{def:hall basis}
  Let $\hallgenerators=\{e_1^{(1)},\dots,e_r^{(1)}\}$ be a set of Lie
  algebra generators for~$\mff = \frcq$. If $u\in\mff$ is a Lie
  product of elements from $\hallgenerators$ then $u$ has {\it weight}
  $\weight(u)=i$ if $u\in\gamma_i(\mff)\setminus\gamma_{i+1}(\mff)$. A
  \emph{Hall basis} (on $\hallgenerators$) for $\mff$ is a well-ordered subset
  $\hallbasis$ of $\mff$, satisfying the following.
\begin{enumerate}
\item $\hallgenerators \subseteq \hallbasis$.
\item If $u,v\in\hallbasis$ then $[u,v]\in\hallbasis$ if and only if 
\begin{equation}\label{hall condition}
u>v \text{ and }(u=[u_1,u_2] \text{ implies } u_2\leq v).
\end{equation}
\item If $w \in\hallbasis\setminus \hallgenerators$ then $w=[u,v]$ for some
  $u,v\in\hallbasis$ satisfying~\eqref{hall condition}.
\item If $u,v\in\hallbasis$ and $\weight(u)>\weight(v)$ then $u>v$.
\end{enumerate}
Elements of $\hallbasis$ are \emph{basic commutators}. For $i\in[c]$,
we set $\hallbasis^{(i)}:=\{h\in\hallbasis:\weight(h)=i\}$ and label
the basic commutators of weight $i$ so that
$\hallbasis^{(i)}=\{e_1^{(i)},\dots,e_{W(i)}^{(i)}\}_<$. Observe that
$g^{(i)}:=[e_2^{(1)},_{i-1}e_1^{(1)}]\in\hallbasis^{(i)}$.
\end{definition}

Choose a Hall basis $\hallbasis$ for $\mff$. It is well-known that the
elements of $\bigcup_{i=2}^{c}\hallbasis^{(i)}$ yield an $\Fq$-basis
for the derived Lie algebra $\mff'$, and that the residue classes of
the elements of $\bigcup_{i=1}^{c-1}\hallbasis^{(i)}$ yield an
$\Fq$-basis for the cocentre $\mff/Z(\mff)$.  Observe that the
commutator matrix
$B(\bfY)\in\Mat\left(\sum_{j=1}^{c-1}W(j),\Fq[\bfY]\right)$ with
respect to $\hallbasis$ is a skew-symmetric matrix of $\Z$-linear
forms in $b=\sum_{j=2}^c W(j)$ variables. We label the variables as
follows. For $k\in [2,c]$ we write
$\bfY^{(k)}=(Y^{(k)}_1,\dots,Y^{(k)}_{W(k)})$. Thus
$\bfY=(\bfY^{(k)})_{k\in [2,c]}$ and
$$B(\bfY) =
\left(\begin{matrix}B_{1,1}(\bfY^{(2)})&B_{1,2}(\bfY^{(3)})&\dots&B_{1,c-1}(\bfY^{(c)})\\
B_{2,1}(\bfY^{(3)})&B_{2,2}(\bfY^{(4)})&\udots&0\\ \vdots&\udots&0&\\
B_{c-1,1}(\bfY^{(c)})&0&&
\end{matrix}\right)$$
\noindent
where $B_{i,j}(\bfY^{(i+j)})$ is the zero matrix  
if $i+j > c$, and for $i,j\in[c-1]$,
\begin{equation}
\label{matrixBij}
B_{i,j} := B_{i,j}(\bfY^{(i+j)}) = -B_{j,i}(\bfY^{(i+j)})^{\trans}\in\Mat(W(i)\times
W(j),\Fq[\bfY^{(i+j)}]).
\end{equation}
For each $k\in[2,c]$, the variables $\bfY^{(k)}$ only occur in the
matrices $B_{i,j}$ with $i+j=k$.  It follows from \cite[Theorem
  1]{StoehrVaughan-Lee/09} that
\begin{itemize}
\item if $j>i$ and $i$ does not divide $j$
then $B_{i,j}$ is \emph{generic}: there are no linear relations
among its entries;
\item if $i=j$ then $B_{i,i}$ is \emph{generic skew-symmetric}: the
  only linear relations between its entries are those resulting from
  the identity $B_{i,i}=-B_{i,i}^{\trans}$.
\end{itemize}
To prove Theorem~$\ref{theorem chi Frc}$ it suffices, by
Theorem~\ref{theorem B}, to show that
\begin{equation}\label{ranks}
\left\{\rk(B(\bfy)) \mid \bfy \in \Fq^b\right\} = 2[n(r,c)]_0.
\end{equation}
The containment $\subseteq$ in \eqref{ranks} is clear, as the rank of
$B(\bfy)$ is clearly bounded from above by $2n(r,c)$. We establish the
containment $\supseteq$ in~\eqref{ranks} by induction on~$c$.  For
$c=1$ there is nothing to prove, and the case $c=2$ is covered by
Proposition~\ref{ch Frc class 2}, so let $c>2$. The induction step is
divided into five steps. To ensure that the induction hypothesis is
applicable, we assume further that~$(r,c)\neq(2,4)$. The statement of
Theorem~$\ref{theorem chi Frc}$ for groups of the form $F_{2,4}(\Fq)$
follows from Proposition~\ref{prop:F24}.

\medskip
\noindent{\bf Step 1:} By the induction hypothesis, we can obtain
every rank in $2[n(r,c-1)]_0$ by setting the `new' variables
$\bfY^{(c)}$ to zero, and arguing as for~$c-1$.

\medskip
\noindent{\bf Step 2:} Let $\rho \in 2[n(r,c-1)+1,n(r,c)]$. If there
exists a vector $\bfy=(\bfy^{(k)})_{k\in[2,c]}\in\Fq^b$, with
$\bfy^{(k)}={\bf 0}$ for $k<c$, satisfying
\begin{enumerate}
\item $\rk(B_{c-i,i}(\bfy^{(c)}))=\min\{W(i),W(c-i)\}=W(i)$ for $i<m =
  \lfloor c/2 \rfloor$,
\item $\rk(B_{c-m,m}(\bfy^{(c)}))=\begin{cases}\rho/2 - \sum_{i=1}^{m-1}
  W(i) &\text{ if }c=2m+1,\\\rho -
  2\sum_{i=1}^{m-1} W(i) &\text{ if }c=2m,\end{cases}$
\end{enumerate}
then $\rk(B(\bfy))=\rho$. Indeed, $B(\bfy)$ is a matrix with nonzero
blocks $B_{i,j}(\bfy^{(c)})$ only in the positions $(i,j)$ where
$i+j=c$. Moreover, apart from the `central block'
$B_{m,m}(\bfy^{(c)})$ if $c=2m$, or `central blocks'
$B_{m,m+1}(\bfy^{(c)})$ and $B_{m+1,m}(\bfy^{(c)})$ if $c=2m+1$, all
blocks have maximal rank.

\medskip
\noindent{\bf Step 3:} We now prove that such a vector $\bfy$
exists. As $(r,c)\neq(2,3)$, $W(i)<W(j)$ whenever~$i<j$. It suffices
to show that, for each $i<c/2$, the matrix $B_{c-i,i}$ has a square
submatrix
$$\wtB_{c-i,i} :=
\wtB_{c-i,i}(\bfY^{(c)})\in\Mat(W(i),\Fq[\bfY^{(c)}]),$$ obtained by
choosing $W(i)$ suitable rows of $B_{c-i,i}$, with the property that
there are no linear relations among the entries $(\wtB_{c-i,i})_{st},
1\leq s \leq t \leq W(i)$, for $1<i<c/2$, and, if $c=2m$, the entries
$(B_{m,m})_{st}$, $1 \leq s \leq t \leq W(m)$.  Indeed, given such
matrices $\wtB_{c-i,i}$, it is easy to construct a vector $\bfy^{(c)}$
such that, for all $i<m$, $\wtB_{c-i,i}(\bfy^{(c)})$ is
lower-unitriangular (and thus, in particular, of maximal rank $W(i)$)
and the central blocks have the required ranks: namely, we set the
diagonal entries of $\widetilde{B}_{i,c-i}$ equal to one, and all the
$(s,t)$-entries of $\widetilde{B}_{i,c-i}$ for $s<t$ equal to zero.

If $c=2m$, the matrix $B_{m,m}$ is generic skew-symmetric,
by~\cite[Theorem 1]{StoehrVaughan-Lee/09}, and so attains every rank
in $2[\lfloor W(m)/2\rfloor]_0$.

\medskip
\noindent{\bf Step 4:} For $i<c/2$ we now exhibit such a submatrix
$\wtB_{c-i,i}$ of $B_{c-i,i}$.  By definition of the commutator
matrix~$B$, the matrix $B_{c-i,i}$ is defined by
$(B_{c-i,i})_{st}=\sum_{k=1}^{W(c)}\lambda_{st}^k Y_k^{(c)}$, where
$[e_s^{(c-i)},e_t^{(i)}]=\sum_{k=1}^{W(c)}\lambda_{st}^ke_k^{(c)}$,
where $s\in W(c-i)$, $t\in W(i)$.  It suffices to find
$S=\{s_{1},\dots,s_{W(i)}\}_<\subseteq[W(c-i)]$, indexing $W(i)$ rows
of $B_{c-i,i}$, such that
\begin{equation}\label{condition B i tilde}
[e_{s_{l}}^{(c-i)},e_t^{(i)}]\in\hallbasis^{(c)} 
\text{ for } l \in [W(i)]
\text{ and } t\in [l,W(i)].
\end{equation}
We then set $\wtB_{c-i,i}:=((B_{c-i,i})_{st})_{s\in S,\, t\in[W(i)]}$.

\medskip
\noindent{\bf Step 5:} To find such a subset $S$ of~$[W(c-i)]$, we
distinguish three cases.

\smallskip
\noindent \emph{Case $(i)$: $i>(c-i)/2$.} Every pair $(e_s^{(c-i)},
e_t^{(i)})\in\hallbasis^{(c-i)}\times \hallbasis^{(i)}$ has the
property that $[e_s^{(c-i)}, e_t^{(i)}]\in\hallbasis^{(c)}$. Indeed,
since $i < c/2$, we deduce that $i<c-i$, so $e_t^{(i)}<e_s^{(c-i)}$.
If $e_s^{(c-i)} = [u_1,u_2]$ for some $u_1,u_2\in\hallbasis$ then
$\weight(u_2)\leq (c-i)/2 <i$ by~\eqref{hall condition}, so
$u_2<e_t^{(i)}$, and hence $[e_s^{(c-i)},
e_t^{(i)}]\in\hallbasis^{(c)}$. Thus every $W(i)$-element subset $S$
of $[W(c-i)]$ satisfies~\eqref{condition B i tilde}.

\smallskip
\noindent \emph{Case $(ii)$: $i<(c-i)/2$.} Let $t\in[W(i)]$.  Since
$i<c-2i$, clearly $e_s^{(c-i)}:=[g^{(c-2i)},e_t^{(i)}]\in
\hallbasis^{(c-i)}$; so, if $v\geq t$ then $[e_s^{(c-i)},
e_v^{(i)}]\in \hallbasis^{(c)}$.  Therefore the set $S$ of indices of
the $W(i)$ elements $[g^{(c-2i)},e_t^{(i)}]$, for $t\in [W(i)]$,
satisfies~\eqref{condition B i tilde}.

\smallskip
\noindent \emph{Case $(iii)$: $i=(c-i)/2$.} Let $t\in[W(i)]$. If $t <
W(i)$ then we set
$e_s^{(2i)}:=[e^{(i)}_{W(i)},e_t^{(i)}]\in\hallbasis^{(2i)}$ and
observe that $[e_s^{(2i)},e_v^{(i)}]\in\hallbasis^{(c)}$ for all $v
\geq t$. If $t=W(i)$ then $[g^{(2i)},e_{W(i)}^{(i)}]\in
\hallbasis^{(c)}$. The set of indices of the $W(i)$ elements
$g^{(2i)}$ and $[e_{W(i)}^{(i)},e_t^{(i)}]$, for $t\in[W(i)-1]$,
satisfies~\eqref{condition B i tilde}.

\medskip
This concludes the proof of Theorem~\ref{theorem chi Frc}.

\subsection{Taketa bounds for $\Frcq$}
The Taketa problem asks for a bound to the derived length $\dl(G)$ of
a finite solvable group $G$ in terms of the number of its character
degrees; see, for example, \cite{Keller/06}.  It is known that
$\dl(G)$ is bounded by a linear function in $|\cd(G)|$.  Isaacs
conjectured that the bound for $p$-groups is logarithmic. It cannot be
better than logarithmic, as the family $(U_n(\Fq))$ shows:
$\dl(U_n(\Fq))=\lceil \log_2(n)\rceil$, but $|\cd(U_n(\Fq))| \sim
n^2/4$; cf.~\cite{Huppert/92}.

Our results exhibit double-logarithmic Taketa bounds for the groups
$\Frcq$, a family of groups of unbounded derived length. Indeed,
$\dl(\Frcq)=\lceil\log_2(c)\rceil$, whereas $n(r,c) \sim r^{\lfloor
  c/2\rfloor}$. Thus $\dl(\Frcq)\leq c_1 \log\log(|\cd(F_{r,c}(\Fq))|) + c_2$ for
suitable constants $c_1,c_2$.

  We also observe that there is a logarithmic bound to the derived
  length of the groups $\Frcq$ in terms of their numbers of class
  sizes. In fact, $|\cs(\Frcq)|=c+1$ (unless both $r$ and $c$ are
  very small), so $\dl(\Frcq)\leq c_3\log(|\cs(\Frcq)|)$.
  The ($\log_q$ of the) class sizes of the groups $U_n(\Fq)$ form an
  interval of length $\binom{n-1}{2}$ (cf.~\cite{VeralopezArregi/94}),
  also yielding a logarithmic bound for this family.

\subsection{Numbers of characters}

Theorem~\ref{theorem chi Frc} describes the support of the character
vectors $\ch(\Frcq)$, showing that the numbers $\ch_{if}(\Frcq)$, for
$i\in [n(r,c)]_0$, are nonzero. We make one observation on the order
of magnitude of the number of characters of maximal
degree~$q^{n(r,c)}$. Define
$$N(r,c):=\sum_{i=1}^cW_r(i) - 2n(r,c).$$

\begin{lemma}
Let $(r,c)\neq(2,3)$, $p>c$ and $q=p^f$.  Then
$$\lim_{q\rightarrow\infty} \frac{\ch_{n(r,c)f}(\Frcq)}{q^{N(r,c)}} = 1.$$
\end{lemma}

\begin{proof} 
  The Lang-Weil estimate (cf.\ \cite{LangWeil/54}) for the number of
  rational points on varieties over finite fields implies that
  $\ch_{n(r,c)f}(\Frcq)\sim q^{r-2n(r,c)}|\Frcq'| = q^{N(r,c)}$.
\end{proof}

By Theorem~\ref{theorem chi Frc}, the smallest degree of a nonlinear
character of $\Frcq$ is~$q=p^f$.  We now count the number of
characters of $\Frcq$ having degree~$q=p^f$, so generalizing
\cite[Theorem~7]{ItoMann/06}.

\begin{proposition}\label{prop:degree p} Let $r\geq 2$, $p>c>2$
  and $q=p^f$. Then
$$\ch_f(\Frcq)= \frac{q^{r-2}(q^r-1)(q^{(r-1)(c-1)+1} +
q^{(r-1)(c-1)} - q^r - 1)}{q^2-1}.$$
\end{proposition}

\begin{proof}
  Let $B(\bfY)$ be the commutator matrix with respect to a Hall basis
  for $\frcq$, and recall the definition \eqref{matrixBij} of the
  matrices~$B_{i,j}$.  For $\bfy\in\Fq^{b}$ we define
$$u(\bfy):= \max\{ i\in [2,c] \mid  
\text{ there exists } j \in [i-1] \text{ such that } B_{j, i-j}(\bfy)
\not= 0 \}.$$ For $s\in[2,c]$, we now compute the quantity $n_s:=
\#\{\bfy\in\Fq^{b} \mid u(\bfy)=s, \, \rk(B(\bfy))=2\}$. This
suffices, as $\ch_f(\Frcq)=q^{r-2}\nu_f(\Frcq)=q^{r-2}\sum_{s=2}^c
n_s$.

For $s=2$, Proposition~\ref{ch Frc class 2} for $k=1$ implies that
$n_2=(q^r-1)(q^{r-1}-1) / (q^2 -1)$.  For $s>2$, we claim
that
\begin{equation}\label{claim ns}
n_s=\frac{q^{(r-1)(s-2)}(q^r-1)(q^{r-1}-1)}{(q-1)}.
\end{equation}
First note that if $u(\bfy)=s$ and $\rk(B(\bfy))=2$, then
$\rk(B_{1,s-1}(\bfy))=1$ and $B_{i,s-i}(\bfy)=0$ for $i \in [2,s-2]$;
see, for example, \cite[Theorem~1]{StoehrVaughan-Lee/09}. In fact,
after a suitable change of basis for $\mathfrak{f}_{r,c}(\F_q)$, we
may assume that $B(\bfy)$ has zero entries everywhere except the first
row and column.  We claim that
\begin{equation*}\label{claim}
\#\{\bfy\in\Fq^{W(s)} \mid \rk(B_{1,s-1}(\bfy))=1,
B_{i,s-i}(\bfy)=0 \text{ for } i \in [2, s-2]\} = \frac{(q^r-1)(q^{r-1}-1)}{(q-1)}.
\end{equation*}
Indeed, there are $q^{r-1}$ ways to fill in a row of $B_{1,s-1}(\bfY)$
so that all other rows are zero. To see this, assume without loss of
generality that this is the first row, and note that exactly $r-1$ of
the Lie products of the form $[e_i^{(s-1)},e_1]$, where $e_i^{(s-1)}$
is a basic commutator of weight $s-1$, are basic, namely the ones of
the form $[e_i, _{s-2}e_1]$ where $i\in [2,r]$. All other Lie products
of the form $[e_i^{(s-1)},e_1]$ are linear combinations of other basic
commutators of weight~$s$. The variables associated to these occur in
some other row of $B_{1,s-1}(\bfY)$, or in some $B_{i,s-i}(\bfY)$ for
$i \ge 2$, and so have the value zero. Up to nonzero scalars, there
are thus $(q^{r-1}-1)/(q-1)$ ways to fill a row without obtaining a
zero row. Every row of $B_{1,s-1}(\bfY)$ is a linear multiple of such
a nonzero row, and only one of the $q^r$ possibilities yields the zero
matrix. This establishes the claim.

We also claim that, for each $\bfy\in\Fq^{W(s)}$ such that
$\rk(B_{1,s-1}(\bfy))=1$ and $B_{i,s-i}(\bfy)=0$ for $i\in[2,s-2]$,
there are $q^{(r-1)(s-2)}$ ways to choose $\bfy'\in\Fq^b$ such that
$\rk(B(\bfy',\bfy))=1$. Indeed, again without loss we may assume that
$B_{1,s-1}(\bfy)$ is supported only on its first row. By the arguments
in the previous paragraph, each of the matrices $B_{1,i}(\bfY)$, for
$i\in[s-2]$, has exactly $r-1$ variables corresponding to basic
commutators in its first row. All other entries in the first row are
linear combinations of variables corresponding to basic commutators
occurring in other rows.  This establishes the claim, and
so~\eqref{claim ns}.

Summing over $s=2,\dots,c$ establishes the result. 
\end{proof}

\subsection{Results on $\Frcq$ for specific values of $r$ and $c$}

We start with a lemma generalizing the opening remarks
of~\cite[Section~3]{ItoMann/06}, thus dealing with the exceptional
parameter values in Theorem~\ref{theorem chi Frc}.

\begin{proposition}\label{prop:F23}
Let $p\geq 5$ and let $q=p^f$. 
$$\ch(F_{2,3}(\Fq)) = (q^2,q^3-1)_f.$$
\end{proposition}

\begin{proof}
  We note that $W_2(1)=2$, $W_2(2)=1$ and $W_2(3)=2$. With respect to
  the Hall basis $\{ e_1^{(1)}, e_2^{(1)}, g^{(2)}, [g^{(2)},
  e_1^{(1)}], [g^{(2)},e_2^{(1)}]\}_<$ for $\mff_{2,3}(\Fq)$ the
  commutator matrix
$$B(\bfY) = \left(\begin{array}{cc|c} &-Y_1&-Y_2\\Y_1&&-Y_3\\\hline Y_2&Y_3& \end{array}\right).$$ The claim follows immediately from Theorem~\ref{theorem B}.
\end{proof}

\begin{proposition} Let $p\geq5$ and let $q = p^f$.
$$\ch(F_{3,3}(\Fq))=(q^3, q(q^3-1)(q^3+q^2+1), q(q^3-1)(q^5+q^4-1),
q^4(q-1)(q^3-q-1))_f.$$
\end{proposition}
\begin{proof}
  Set $F = F_{3,3}(\Fq)$.  We note that $W_3(1)=3, W_3(2)=3$ and
  $W_3(3)=8$.  With respect to a Hall basis for $\mff_{3,3}(\Fq)$, the
  commutator matrix $B(\bfY)$ is
$$\left(\begin{array}{ccc|ccc}
    &-Y_1&-Y_2&-Y_4&-Y_5&-Y_6\\
    Y_1&&-Y_3&-Y_7&-Y_8&-Y_9\\
    Y_2&Y_3&&Y_6+Y_8&-Y_{10}&-Y_{11}\\
    \hline Y_4&Y_7&-Y_6-Y_8\\
    Y_5&Y_8&Y_{10}&&&\\
    Y_6&Y_9&Y_{11}\end{array}\right)
= \left(\begin{array}{c|c} B_{11}(\bfY^{(2)}) &
    B_{12}(\bfY^{(3)})\\\hline B_{21}(\bfY^{(3)})& \end{array}\right).$$ 
It suffices to prove our claim for $\ch_{2f}(F)$. Indeed, 
$n(3,3)=3$, the claim for $\ch_0(F)$ is trivial, and that for
$\ch_f(F)$ follows from Proposition~\ref{prop:degree p}. Furthermore, the
class number $k(F)=\sum_{i=0}^3q^{3-2i}\nu_{if}(F)$ is 
$q^9-2q^8-q^6-q^5$ by Theorem \ref{theorem conjugacy classes}.

We claim that there are $q(q^3-1)(q^5+q^4-1)$ vectors
$\bfy\in\Fq^{11}$ such that $\rk(B(\bfy))=4$. For such $\bfy$ we
distinguish whether $\rk(B_{12}(\bfy))=1$ or $\rk(B_{12}(\bfy))=2$. In
the former case, by \eqref{claim} in the proof of
Proposition~\ref{prop:degree p}, there are $(q+1)(q^3-1)$ vectors
$(y_4,\dots,y_{11}) \in \Fq^8$ yielding
$\rk(B_{12}(y_4,\dots,y_{11}))=1$, and for each of these there are
$q^3-q^2$ vectors $(y_1,y_2,y_3)\in\Fq^3$ such that
$\rk(B(y_1,\dots,y_{11}))=4$. Thus
\begin{equation*}
\#\{\bfy\in\Fq^{11} \mid \rk(B(\bfy))=4, \rk(B_{12}(\bfy))=1\} =
q^2(q^2-1)(q^3-1).
\end{equation*}
On the other hand, the set $N:=\{ \bfy\in\Fq^8 \mid
\rk(B_{12}(\bfy))=2\}$ has cardinality $q(q^3-1)(q^3+q^2 -1)$;
cf.~\cite{Bender/74}. As every vector in $N$ gives rise to $q^3$
matrices $B(\bfy)$ of rank $4$,
\begin{equation*}
\#\{\bfy\in\Fq^{11} \mid \rk(B(\bfy))=4, \rk(B_{12}(\bfy))=2\} =
 q^4(q^3-1)(q^3+q^2-1)
\end{equation*}
and thus
$$\nu_{2f}(F)=q^2(q^2-1)(q^3-1) + q^4(q^3-1)(q^3+q^2-1)=
q^2(q^3-1)(q^5+q^4-1),$$ which yields the claimed quantity
for~$\ch_{2f}(F)=q^{3-4}\nu_{2f}(F)$. 
\end{proof}

\noindent 
We obtain the following generalization of~\cite[Lemma~14]{ItoMann/06}.
\begin{proposition}\label{prop:F24}
Let $p\geq 5$ and let $q = p^f$.
$$\ch(F_{2,4}(\Fq)) = (q^2,q^4+q^3-q^2-1,q^4-q^2-q+1)_f.$$
\end{proposition}

\begin{proof} 
  Note that $n(2,4)=2$. The formula for $\ch_f(F_{2,4}(\Fq))$ is given
  by Proposition~\ref{prop:degree p} and the class number $k(F_{2,4}(\Fq))$
  is given by Theorem~\ref{theorem conjugacy classes}.
\end{proof}

\begin{proposition}Let $p\geq 7$ and let $q = p^f$.
The nonzero values of $\ch_i(F_{2,5}(\Fq))$ are given as follows.
\begin{center}
\begin{tabular}{r|l}
i&$\ch_i(F_{2,5}(\Fq))$\\ \hline
$0$&$q^2$ \\ 
$f$&$(q-1)(q^4+ 2q^3+2q^2+q+1)$\\
$2f$&$(q-1)(q^7+2q^6+3q^5+2q^4+q^3-q-1)$ \\
$3f$&$q^2(q^2-1)(q^4-q-1)$ 
\end{tabular}
\end{center}
\end{proposition}

\begin{proof} 
  Set $F=F_{2,5}(\Fq)$, and note that $n(2,5)=3$. The statement about
  $\ch_0(F)$ is trivial; the claim about $\ch_f(F)$
  is a special case of Proposition~\ref{prop:degree p}. By 
  Theorem~\ref{theorem conjugacy classes}, $k(F)=2q^8 + q^7 - q^5 -
  q^4$, so it suffices to compute, for instance,~$\ch_{3f}(F)$. 

We now describe a Hall basis for $\mathfrak{f}_{2,5}(\Fq)$.  We choose
Lie generators $x$ and $y$, where $y < x$, and omit Lie brackets in
left-normed Lie products, so, for example, $[[x,y],x]$ is represented
by~$xyx$.  It is easily verified that the following elements form a
Hall basis.

\begin{center}
\begin{tabular}{r|l}
$j$& Basis elements of weight $j$\\ \hline
$1$& $y,x$\\
$2$& $xy$\\
$3$& $xyy, \;xyx$\\
$4$& $xyyy, \;xyyx, \;xyxx$\\
$5$& $xyyyx, \; xyyxx, \;xyxxx,  \;xyyyy, \;(xyx)(xy), \;(xyy)(xy)$
\end{tabular}
\end{center}
With respect to this basis, the commutator matrix 
$$B(\bfY) = \left(
  \begin{array}{cc|c|cc|ccc}&-Y_1&-Y_2&-Y_4&-Y_5&-Y_{10}&-Y_{11}-Y_7&-Y_{12}-Y_8\\
    Y_1&&-Y_3&-Y_5&-Y_6&-Y_7&-Y_8&-Y_9\\ \hline
    Y_2&Y_3&&-Y_{11}&-Y_{12}&&&\\ \hline Y_4&Y_5&Y_{11}&&&&&\\
    Y_5&Y_6&Y_{12}&&&&&\\\hline Y_{10}&Y_7&&&&&&\\
    Y_{11}+Y_7&Y_8&&&&&&\\ Y_{12}+Y_8&Y_{9}&&&&&&
       \end{array}\right).$$

It suffices to prove that $\nu_{3f}(F)=q^6(q^2-1)(q^4-q-1)$. If
$\bfy\in\Fq^{12}$ and $\rk(B(\bfy))=6$
then~$(y_{11},y_{12})\neq\{{\bf 0}\}$.  Fix
$(y_{11},y_{12})\in\Fq^2\setminus\{{\bf 0}\}$. It is easily checked that
$$\#\{\bfy=(y_7,\dots,y_{10})\in\Fq^4 \mid \rk(B_{14}(\bfy))=1\} =
q(q+1).$$ Given $\bfy=(y_7,\dots,y_{10})\in\Fq^4$ with
$\rk(B_{14}(\bfy))=1$, there are $q^5(q-1)$ ways to choose
$(y_1,\dots,y_6)\in\Fq^6$ such that $\rk(B(y_1,\dots,y_{12}))=6$.
Similarly,
$$\#\{\bfy=(y_7,\dots,y_{10})\in\Fq^4 \mid \rk(B_{14}(\bfy))=2\} =
q^4-q(q+1).$$ Given $\bfy=(y_7,\dots,y_{10})\in \Fq^4$ with
$\rk(B_{14}(\bfy))=2$, there are $q^6$ ways to choose
$(y_1,\dots,y_6)\in\Fq^6$ such that $\rk(B(y_1,\dots,y_{12}))=6$.
Thus $$\nu_{3f}(F)=(q^2 - 1)\left(q^5(q-1)\cdot q(q+1) + q^6\cdot
(q^4-(q^2+q))\right) = q^6(q^2-1)(q^4-q-1)$$ as claimed.
\end{proof}			

\begin{remark} \label{rem:negative coeff} We note that
  $\ch_{3f}(F_{2,5}(\Fq))$ is given by a polynomial in $q$ and its
  expansion in $v:=q-1$ has both positive and negative
  coefficients. Indeed
\begin{equation*}
  \ch_{3f}(F_{2,5}(\Fq)) = v(v+2)(v+1)^2(v^4+4v^3+6q^2 +3v-1).
\end{equation*} 
We observe this phenomenon only for the family of groups
$F_{2,5}(\Fq)$, for $p \geq 7$; in all other families we considered
the corresponding coefficients are nonnegative.
\end{remark}

\begin{acknowledgements} 
We acknowledge support from the Alexander von Humboldt Foundation, the
EPSRC, the Marsden Fund of New Zealand and the Royal Society.  We
thank I.~M.\ Isaacs, L.~G.\ Kov\'acs and Avinoam Mann for helpful
discussions.

\end{acknowledgements}


\def\cprime{$'$}
\providecommand{\bysame}{\leavevmode\hbox to3em{\hrulefill}\thinspace}
\providecommand{\MR}{\relax\ifhmode\unskip\space\fi MR }
\providecommand{\MRhref}[2]{%
  \href{http://www.ams.org/mathscinet-getitem?mr=#1}{#2}
}
\providecommand{\href}[2]{#2}

\end{document}